\documentclass[11pt, oneside]{amsart}   	
\usepackage{amsmath,amsthm,amscd,amsfonts, amssymb, mathrsfs}
\usepackage{geometry}                		
\usepackage{graphicx}				
\usepackage{amssymb}
\newcommand{\sect}[1]{\section{#1}\setcounter{equation}{0}}
\newcommand{\subsect}[1]{\subsection{#1}}

\font\mbn=msbm10 scaled \magstep1
\font\mbs=msbm7 scaled \magstep1
\font\mbss=msbm5 scaled \magstep1
\newfam\mbff
\textfont\mbff=\mbn
\scriptfont\mbff=\mbs
\scriptscriptfont\mbff=\mbss   

\newcommand{\Di}      {\mathbb{D}}
\newcommand{\RR}       { \mathbb{R}}

\newcommand{\N}       { \mathbb{N}}
  
\newcommand\Co           {{\mathbb C}}

\newtheorem{Th}{Theorem}[section]
\newtheorem{Lm}[Th]{Lemma}
\newtheorem{C}[Th]{Corollary}

\newtheorem{Prop}[Th]{Proposition}
\newtheorem{R}[Th]{Remark}
\newtheorem{Problem}[Th]{Problem}

\newtheorem{E}[Th]{Example}

\newtheorem*{Th A}{Theorem A}
\newtheorem*{Th B}{Theorem B}
\begin{document}

\title[On Nonlinear Rudin-Carleson Type Theorems]{On Nonlinear Rudin-Carleson Type Theorems}
\author{Alexander Brudnyi}
\address{Department of Mathematics and Statistics\newline
\hspace*{1em} University of Calgary\newline
\hspace*{1em} Calgary, Alberta, Canada\newline
\hspace*{1em} T2N 1N4}
\email{abrudnyi@ucalgary.ca}


\thanks{Research is supported in part by NSERC}

\begin{abstract} 
In this paper we study nonlinear interpolation problems for interpolation and peak-interpolation sets of function algebras. The subject goes back to the classical Rudin-Carleson interpolation theorem. In particular, we prove the following nonlinear version of this theorem:
 Let $\bar{\mathbb D}\subset \mathbb C$ be the closed unit disk, $\mathbb T\subset\bar{\mathbb D}$ the unit circle, $S\subset\mathbb T$ a closed subset of Lebesgue measure zero and $M$ a connected complex manifold.
 Then for every  continuous $M$-valued  map $f$ on $S$ there exists a continuous $M$-valued map $g$ on $\bar{\mathbb D}$ holomorphic on its interior  such that $g|_S=f$. We also consider similar interpolation problems for continuous maps
 $f: S\rightarrow\bar M$, where $\bar M$ is a complex manifold with boundary $\partial M$ and interior $M$. Assuming that $f(S)\cap\partial M\ne\emptyset$ we are looking for holomorphic extensions $g$  of $f$ such that $g(\bar{\mathbb D}\setminus S)\subset M$. 
  \end{abstract}

\date{}

\maketitle

\sect{Formulation of Main Results}
\noindent {\bf 1.1.} 
Let $A$ be a {\em uniform} algebra on a compact Hausdorff space $X$, i.e.,
a closed unital subalgebra of the Banach algebra $C(X)$ of complex continuous functions on $X$ equipped with the norm $\|f\|_{C(X)}:=\max_X |f|$ separating points of $X$. (For the theory of uniform algebras see, e.g., the book \cite{G}.) \smallskip

 A compact subset $S\subset X$ is said to be {\em interpolation set} for $A$ if the restriction  to $S$ maps $A$  onto $C(S)$. The number
\begin{equation}\label{e1}
c_A(S):=\sup_{f\in C(S),\, \|f\|_{C(S)}=1} \inf\{\|F\|_{C(X)}\, :\, F\in A,\ F|_S=f\}
\end{equation}
(finite by the Banach open mapping theorem) is called the {\em interpolation constant } for $S$. 

In this paper we consider interpolation problems for continuous maps of $S$ in complex manifolds.
To formulate our results we require several definitions.

\smallskip

The {\em maximal ideal space} $\mathfrak M(A)$ is the set of all nontrivial complex homomorphisms of $A$.
It is a compact subset of the closed unit ball of the dual space $A^*$ equipped with the weak$^*$ topology. The {\em Gelfand transform} $\hat{\,}: A\rightarrow C(\mathfrak M(A))$, $\hat{a}(\varphi):=\varphi(a)$, maps $A$ isometrically onto a uniform subalgebra on  $\mathfrak M(A)$ and its transpose embeds $X$ into $\mathfrak M(A)$. Without loss of generality we will identify $A$ with its image under $\hat{\,}$
and $X$ with its image under the embedding. \smallskip

A complex manifold $ M$ is said to be {\em Oka} if every holomorphic map $f: K\rightarrow  M$ from a neighbourhood of  a compact convex set $K \subset\Co^n$, $n\in\N$, can be approximated uniformly on $K$ by entire maps $\mathbb C^n\rightarrow M$.

The class of Oka manifolds includes, in particular, complex homogeneous manifolds, complements in $\mathbb C^n$, $n>1$, of complex algebraic subvarieties of codimension $\ge 2$ and of compact polynomially convex sets, Hopf manifolds (i.e., nonramified holomorphic quotients of $\mathbb C^n\setminus\{0\}$). Also,  holomorphic fibre bundles whose bases and fibres are Oka manifolds are Oka manifolds as well. (We refer to the book \cite{F1} and the paper \cite{K} for other examples and basic results of the theory of Oka manifolds.)\smallskip

In what follows, $C(X,Y)$ stands for the set of continuous maps  between topological spaces $X$ and $Y$.  For a uniform algebra $A$ on $X$ and a subspace $M\subset\mathbb C^n$, we denote by $A(X,M)\subset C(X,M)$ the set of maps with coordinates in $A$. For $V\subset C(X,Y)$ and $S\subset X$ the {\em trace space} $V|_S\, (\subset C(S,Y))$ consists of restrictions of maps in $V$ to $S$.\smallskip

The following result is a particular case of \cite[Thm.\,1.4]{Br}.
\begin{Th}\label{te0}
Let $M\subset\mathbb C^n$ be a complex regular submanifold\,\footnote{ I.e., $M$ is equipped with the induced topology.} and an Oka manifold and $S\subset \mathfrak M(A)$ be an interpolation set for the uniform algebra $A$. Then
\[
A(\mathfrak M(A),M)|_S=C(\mathfrak M(A),M)|_S.
\]
\end{Th}

In other words, under the above conditions a map $f\in C(S,M)$ extends to a $g\in A(\mathfrak M(A),M)$ if and only if it extends to a map from $C(\mathfrak M(A), M)$.
In the next result, we  show that for totally disconnected interpolation sets (such as in the Rudin-Carleson theorem)  similar interpolation problems are always solvable in a more general setting.\smallskip

Let  $A$ be a uniform algebra on $X$. For a family $\mathcal F=\{f_1,\dots, f_n\}\subset A$ we denote by $A_{\mathcal F}\subset A$ the closed unital subalgebra generated by $f_1,\dots, f_n$. The maximal ideal space $\mathfrak M(A_\mathcal F)$ can be naturally identified with the polynomially convex hull of the compact set
$\{F(x):=(f_1(x),\dots, f_n(x))\in \mathbb C^n\, :\, x\in X\}\subset\mathbb C^n$,  the {\em joint spectrum} of $\mathcal F$. 

Let $M$ be a complex manifold and $g$ be a holomorphic map into $M$ defined on a neighbourhood of $\mathfrak M(A_\mathcal F)$. The continuous map $F^*g:=g\circ F: \mathfrak M(A)\rightarrow M$ is said to be {\em holomorphic}; the set of such maps is denoted by
$\mathcal O_{\mathcal F}(\mathfrak M(A), M)$. Note that if $M\subset\Co^N$, $N\in\N$, then $\mathcal O_{\mathcal F}(\mathfrak M(A), M)\subset A(\mathfrak M(A), M)$.

\begin{Th}\label{te1}
Let $S\subset X$ be a totally disconnected interpolation set for $A$ and $M$ be a connected complex manifold.
\begin{itemize}
\item[(a)] If $c_A(S)=1$, then for every   $f\in C(S, M)$ there exists a map  $g\in \mathcal O_{\mathcal F}(\mathfrak M(A),M)$, where $|\mathcal F|={\rm dim}_{\mathbb C} M$, such that
$g|_{S}=f$.\smallskip
\item[(b)] Suppose $M$ is an Oka manifold and $c_A(S)>1$. Let $K\Subset M$ be an open relatively compact subset.  There is a subset $L\Subset M$ containing $K$ such that for every $f\in C(S, K)$ there exists a map $g\in \mathcal O_{\mathcal F}(\mathfrak M(A),M)$, where $|\mathcal F|={\rm dim}_{\mathbb C} M$, such that $g(\mathfrak M(A)\subset L$ and $g|_S=f$.
\end{itemize}
Here $|\mathcal F|$ stands for the cardinality of $\mathcal F$.
\end{Th}
\begin{E}\label{e1.2}
{\rm (1)  Let $A(\mathbb D^n)\subset C(\bar{\mathbb D}^n)$ be the uniform algebra of continuous functions on the closure $\bar{\mathbb D}^n$ of the open unit polydisk $\mathbb D^n\subset\mathbb C^n$ holomorphic on $\mathbb D^n$. Let $\mathbb T^n\subset \bar{\mathbb D}^n$ be the boundary torus. It is proved in \cite{RS} that every compact subset $S\subset\mathbb T^n$ of zero $1$-dimensional Hausdorff measure  is an interpolation set for $A(\mathbb D^n)$ with the interpolation constant $1$. Since such $S$ is totally disconnected, Theorem \ref{te1} implies the following extension of the Rudin-Carleson theorem (see \cite{R}, \cite{S}):}\smallskip

  Let $M$ be a connected complex manifold. For every $f\in C(S,M)$ there exists a $g\in C(\bar{\mathbb D}^n,M)$ holomorphic on $\mathbb D^n$ such that $g|_S=f$.\smallskip

\noindent {\rm (2) Let $Z$ be a connected complex manifold such that the algebra $H^\infty(Z)$ of bounded holomorphic functions on $Z$ separates  points. A sequence $S=\{s_n\}_{n\in\N}\subset Z$ is called {\em interpolating} for $H^\infty(Z)$ if $H^\infty(Z)|_S$ coincides with the Banach space of bounded complex-valued functions on $S$ equipped with supremum norm. The interpolation constant $c(S)$ is defined similarly to \eqref{e1}.
Let $X$ be the closure of $Z$ in the maximal ideal space $\mathfrak M(H^\infty(Z))$. We identify $H^\infty(Z)$ with its image in $C(\mathfrak M(H^\infty(Z)))$ under the Gelfand transform. Then the closure $\bar S\subset X$ of $S$ is an interpolation set for $H^\infty(Z)$ with the interpolation constant $c(S)$.  Moreover, $\bar S$ is homeomorphic to the Stone-\v{C}ech compactification  of  $\N$ and, hence, is totally disconnected. Now, Theorem \ref{te1} implies that
$S$ is also an interpolating sequence for bounded holomorphic maps into connected Oka manifolds:}\smallskip

Let $M$ be a connected Oka manifold. Let $f:S\rightarrow M$ be a bounded map with image in a compact subset $K\subset M$. Then there exist a compact subset $L\subset M$ depending on $M,K$ and $c(S)$ only and a holomorphic map $g: Z\rightarrow M$ with image in $L$ such that  $g|_S=f$.\smallskip
\end{E}

\noindent {\bf 1.2.} In this part we consider nonlinear interpolation problems for peak-interpolation sets.
(E.g., interpolation sets in the Rudin-Carleson theorem and in \cite{RS} are peak-interpolation.)\smallskip

Recall that a compact subset $S\subset X$ is said to be {\em peak-interpolation} for a uniform algebra $A$ on $X$ if every nonidentically zero function $f\in C(S)$ extends to a $g\in A$ that satisfies
\begin{equation}\label{e1.2}
|g(x)|<\max_S |f|\qquad \forall x\in S^c:=X\setminus S.
\end{equation}
Equivalently, every $f\in C(S, \bar{\mathbb D})$ with $f(S)\cap\mathbb T\ne\emptyset$ extends to a  $g\in A$ that satisfies $g( S^c )\subset\mathbb D$. Thus, for peak-interpolaton sets it is naturally to consider interpolation problems for maps into complex manifolds with boundaries.\smallskip

Let $B$ be a complex Banach space and $Y\subset B$ be a complex Banach submanifold. For a uniform algebra $A$ on $X$ we denote by $A(X,Y)$ the set of maps $f\in C(X,Y)$ such that
$\varphi(f)\in A$ for every $\varphi\in B^*$. We are interesting in submanifolds  subject to the following definition:\smallskip

A complex Banach submanifold $\bar M\subset B$ with boundary $\partial M$ and interior $M$ is said to be {\em universal} if for every compact Hausdorff space $X$, a uniform algebra $A\subset C(X)$ and a peak-interpolation set $S\subset X$ for $A$ the following holds:

Every $f\in C(S, \bar M)$ with $f(S)\cap\partial M\ne\emptyset$ extends to a map $g\in A(X,\bar M)$ such that
$g(S^c)\subset M$.\smallskip

The class of universal manifolds has the following properties.
\begin{Prop}\label{prop1.4}
\begin{itemize}
\item[(1)] Direct product of universal manifolds is a universal manifold.

\item[(2)] The set of universal submanifolds of a complex Banach space $B$ is invariant with respect to the action of the group of invertible affine transformations of $B$.

\item[(3)] If $M\subset \Co^n$ is a universal submanifold and $F: \Co^n\rightarrow\Co^m$ is a holomorphic embedding, then $F(M)$ is a universal submanifold of $\Co^m$.

\item[(4)] Every paracompact universal manifold is contractible.
\end{itemize}
\end{Prop}
It is known that a closed ball of a complex Banach space $B$ is universal, see \cite{S}.
Our next result generalizes this fact.\smallskip

Let $p_M: B\rightarrow [0,\infty)$ be the Minkowski functional of an open absorbing subset $M$ of a complex Banach space $B$, i.e.,
\begin{equation}\label{equ1.3}
p_M(v):=\inf_ {tv \in M,\, t > 0 } \, \frac{1}{t} .
\end{equation}
Then $p_M$ is homogeneous, i.e., $p_M(r v)=r\, p_M(v)$ for all $r\in\mathbb R_+$, $v\in B$.
\begin{Prop}\label{prop1.5}
Suppose $M$ satisfies  the condition:
\begin{equation}\label{equ1.4}
\mbox{If $ v \in \bar M$, the closure of $M$, then the entire segment $ [ 0 , v )$ lies in $M$.}
\end{equation}
Then $p_M$ is a continuous function.

Conversely, if $p:B\rightarrow [0,\infty)$ is a continuous homogeneous function, then $M:=\{v\in B\, :\, p(v)<1\}$ is an open absorbing set satisfying \eqref{equ1.4} and $p=p_M$.
 \end{Prop}
\noindent Note that if $M$ satisfies \eqref{equ1.4} , then $M=\{v\in B\, :\, p_{M}(v)<1\}$ and  $\partial M=\{v\in B\, :\, p_{M}(v)=1\}$ (the boundary of $M$). Moreover, $\bar M$ is a complex Banach manifold with boundary modelled on $B$ (in fact, $\bar M\setminus p_M^{-1}(0)$ is homeomorphic to $(0,1]\times\partial M$).

For instance, an open convex neighbourhood of $0\in B$ satisfies \eqref{equ1.4}. In this case the function $p_M$ is subadditive (i.e., $p_M(v+v')\le p_M(v)+p_M(v')$). Also, every star body $M\subset\Co^n$ containing $0$ satisfies \eqref{equ1.4}. If, in addition, such $M$ is bounded and $0$ is an interior point of its kernel, then $p_M$ is a Lipschitz function, see, e.g., \cite{T}.
\begin{Th}\label{te1.4}
Suppose $M\subset B$ is an open absorbing subset satisfying \eqref{equ1.4}. Then  $\bar M$ is a universal manifold.
\end{Th}
\begin{R}\label{rem1.7}
{\rm Since $p_M$ is a homogeneous function, the theorem can be restated as follows:}\smallskip

Given a uniform algebra $A$ on $X$ and a peak-interpolation set $S\subset X$  for $A$  every $f\in C(S,\bar M)$  such that $p_M\circ f\not\equiv 0$ has an extension $g\in A(X,\bar M)$ that satisfies
\[
p_M(g(x))<\max_{y\in S}\,p_M(f(y))\qquad \forall x\in S^c.
\]
\end{R}

For a subset $K\subset B$  we denote by ${\rm co}(K)$ the convex hull of $K$ (-- the minimal convex subset of $B$ containing $K$). Also, by $[K]_\varepsilon\subset B$ we denote the open $\varepsilon$-neighbourhood of $K$:
\[
[K]_\varepsilon:=\{v\in B\, :\, \inf_{v'\in K}\, \|v-v'\|_B<\varepsilon\}.
\]
\begin{C}\label{cor1.7}
Let $A$ be a uniform algebra on $X$ and $S\subset X$ be a peak-interpolation set for $A$. Let $f\in C(S,B)$.
For every $\varepsilon>0$ there is a  $g_\varepsilon \in A(X, [{\rm co}(f(S))]_\varepsilon)$ such that $g_\varepsilon|S=f$. \\
Moreover, if $B=\Co^n$ and ${\rm co}(f(S))$ has a nonempty interior $({\rm co}(f(S)))^\circ$, then there is a $g\in A(X,{\rm co}(f(S)))$ extending $f$ such that $g(S^c)\subset ({\rm co}(f(S)))^\circ$.
\end{C}
For the last statement, note that since $f(S)\subset\Co^n$ is compact, ${\rm co}(f(S))$ is compact as well by the Caratheodory theorem, and  $({\rm co}(f(S)))^\circ\ne\emptyset$ provided that $f(S)$ contains $2n+1$ linearly independent vectors over $\RR$.\smallskip

\begin{R}\label{rem1.8}
{\rm Inspired by Theorem \ref{te1}, one can consider an analogous interpolation problem for totally disconnected peak-interpolation sets.}

\begin{Problem}\label{prob1.9}
{\rm 
Let $\bar M$ be a domain with  boundary in a complex manifold $N$. For what $\bar M$ the following holds:\smallskip

\noindent (*) For every uniform algebra $A$ on $X$, a totally disconnected peak-interpolation set $S\subset X$ for $A$ and a map $f\in C(S, \bar M)$ with $f(S)\cap\partial M\ne\emptyset$  there are a subset $\mathcal F\subset A$ with $|\mathcal F|={\rm dim}_\Co \, N$ and a map
$g\in \mathcal O_\mathcal F(\mathfrak M(A),N)$ such that $g|_S=f$ and $g(S^c)\subset M$?}
\end{Problem}
{\rm Theorem \ref{te1} (a) asserts that every $\bar M$ is near-optimal meaning that for every open neighbourhood $O\subset N$ of $\bar M$ there exists a map $g\in \mathcal O_\mathcal F (\mathfrak M(A),O)$ with $g|_S=f$.  We conjecture that for $\bar M$ with a `nice' boundary (e.g., for strongly pseudoconvex domains $\bar M\subset\Co^n$) such near-optimal $g$ can be deformed to obtain the one satisfying condition (*).}\smallskip

{\rm Clearly, if $\bar M_i\subset N_i$, $i=1,2$, satisfy (*), then $\bar M_1\times\bar M_2\subset N_1\times N_2$  satisfies (*) as well. Also, universal submanifolds $\bar M\subset \Co^n$ satisfy (*). 
The following result gives an example of nonuniversal $\bar M$ satisfying (*) (cf.  Proposition \ref{prop1.4}\,(4)).}
\begin{Th}\label{te1.10}
Let $\bar M$ be a connected Riemann surface with boundary embedded in a Riemann surface $N$ such that inclusion $\bar M\hookrightarrow N$ is homotopy equivalence. Then $\bar M$ satisfies condition (*).
\end{Th}
{\rm Note that every connected Riemann surface with boundary $\bar M$ can be embedded in its double $W$. Then there is an open neighbourhood $N\subset W$ of $\bar M$  that satisfies the hypothesis of the theorem.}
\end{R}
\sect{Proofs of Theorems \ref{te1} and \ref{te1.10}}
\subsect{Proof of Theorem \ref{te1}}
(a) Due to the main theorem of \cite{FS} there is a finite locally biholomorphic surjective map $h:\mathbb D^n\rightarrow M$, where $n={\rm dim}_\mathbb C M$.  Since $f(S)\subset M$ is compact and $h$ is locally biholomorphic, there exist a finite open cover $(U_i)_{1\le i\le k}$ of $f(S)$ and  holomorphic maps $\tilde h_i:U_i\rightarrow \mathbb D^n$ such that $h\circ h_i={\rm id}_{U_i}$, $1\le i\le k$. Consider the
finite open cover $\mathfrak V= (f^{-1}(U_i))_{1\le i\le k}$ of $S$. Since $S$ is compact and totally disconnected, its covering dimension is zero (for basic results of the dimension theory, see, e.g., \cite{N}). In particular, there is a refinement $(W_s)_{1\le s\le m}$ of $\mathfrak V$  by clopen pairwise disjoint subsets. Let $\tau:\{1,\dots,m\}\rightarrow \{1,\dots, k\}$ be the refinement map, i.e., $W_s\subset f^{-1}(U_{\tau(s)})$, $1\le s\le m$. Let us define a map $\tilde f:S\rightarrow\mathbb D^n$ by the formula
\begin{equation}\label{e2.1}
\tilde f(x):=\tilde h_{\tau(s)}(f(x)),\qquad x\in W_s,\quad 1\le s\le m.
\end{equation}
Then, $\tilde f\in C(S,\mathbb D^n)$ and $h\circ \tilde f= f$. Since $S$ is an interpolation set with $c_A(S)=1$, we can extend coordinates of $\tilde f$ to get a continuous map
$\tilde g: \mathfrak M(A)\rightarrow \mathbb D^n$ with coordinates in $A$ such that $\tilde g|_{S}=\tilde f$.  Let $\mathcal F$ be the family of coordinates of $\tilde g$.
Then $g:=h\circ\tilde g\in\mathcal O_\mathcal F(\mathfrak M(A),M)$ is the required map interpolating $f$ on $S$.\medskip

\noindent (b) According to \cite[Thm.\,1.1]{F2} there is a surjective holomorphic map $h:\mathbb C^n\rightarrow M$,  where $n={\rm dim}_\mathbb C M$, such that for every $x\in M$ there are an open neighbourhood $U_x\Subset M$ of $x$ and a holomorphic map $\tilde h_x :U_x\rightarrow\mathbb C^n$ such that $h\circ h_x={\rm id}_{U_x}$. Let $(U_{x_i})_{1\le i\le k}$ be a finite open cover of the compact set $\bar K$ (-- the closure of $K$). Then $V:=\cup_{ i=1}^k\,\tilde h_{x_i}(U_{x_i})$ and $\tilde K:= V\cap  h^{-1}(K)$ are open relatively compact subsets of $\mathbb C^n$. Let $D_{\tilde K}\Subset\mathbb C^n$ be the minimal open polydisk centered at $0$ containing $\tilde K$ and $c_S(A)D_{\tilde K}$ be the dilation of $D_{\tilde K}$ with scalar factor $c_S(A)$. We define
\begin{equation}\label{e2.2}
L:=h(c_S(A)D_{\tilde K}).
\end{equation}
Let $f\in C(S, K)$. Then as in the proof of part (a) of the theorem we construct a map $\tilde f\in C(S,D_{\tilde K})$  such that $f=h\circ\tilde f$.  By the definition of an interpolation set, there is a map
$\tilde g\in C( \mathfrak M(A), c_S(A)D_{\tilde K})$ with coordinates in $A$ such that $\tilde g|_{S}=\tilde f$.  Let $\mathcal F$ be the family of coordinates  of $\tilde g$.
Then $g:=h\circ\tilde g\in\mathcal O_\mathcal F(\mathfrak M(A),M)$ satisfies $g(\mathfrak M(A))\subset L$ and $g|_S=f$, as required.\hfill $\Box$
\subsect{Proof of Theorem \ref{te1.10}}
  Let $r: N_u\rightarrow N$ be the universal covering of $M$. Then $\bar M_u:=r^{-1}(\bar M)$ is the universal covering  of $\bar M$, and $\partial{M_u}:=r^{-1}(\partial M)$ and $M_u:=r^{-1}(M)$ are the boundary and the interior of $\bar M_u$. By our condition, $M_u$ and $N_u$ are biholomorphic to $\mathbb D$.  Thus, without loss of generality, we assume that $N_u$ coincides with $\Di$. Then $M_u$ is a simply connected domain in $\Di$ and (as $\bar M_u$ is a manifold with boundary) $\partial M_u$ is homeomorphic to a one-dimensional manifold and there is a neighbourhood of $\partial M_u$ in  $\bar M_u$ homeomorphic to $\partial M_u\times (0,1]$. In particular, each open arc in $\partial M_u$ is a {\em free boundary arc}, see \cite[Sec.\,3.1]{P}. Hence, a biholomorphic map
 $h: M_u\rightarrow\Di$ extends to an injective continuous map $\bar M_u\rightarrow\bar\Di$, see \cite[Thm.\,3.1]{P}. 
 
 Let $A$ be a uniform algebra on $X$ and $S\subset X$ be a totally disconnected peak-interpolation set for $A$. Let $f\in C(S,\bar M)$ be such that $f(S)\cap\partial M\ne\emptyset$. Since $r$ is locally biholomorphic, as in the proof of Theorem \ref{te1} we can construct a map $\tilde f\in C(S,\bar M_u)$ such that $f=r\circ\tilde f$. Consider the map $h\circ\tilde f\in C(S,\bar\Di)$. By our hypothesis, $(h\circ \tilde f)(S)\cap\mathbb T\ne\emptyset$. Then by the definition of the peak-interpolation set, there exists $g'\in A$ such that $g'|_S=h\circ \tilde f$ and $g'(S^c)\subset\Di$. In turn, $\tilde g:=h^{-1}\circ g'$ maps $S^c$ in $M_u$ and coincides with $\tilde f$ on $S$. 

Let us prove that $\tilde g\in A$.

 In fact, let $K:=g'(X)$. Then $K$ is a compact subset of $L\cup\Di$, where $L:=(h\circ \tilde f)(S)$. Moreover, since $\tilde f(S)$ is a compact subset of $\partial M_u$, the open set $\mathbb T\setminus L\subset\mathbb T$ is nonvoid. Let $z\in\mathbb T\setminus L$.  Then there is an open disk $D$ centered at $z$ such that $D\cap K=\emptyset$. In particular, for a point $z'\in D\setminus\bar\Di$ sufficiently close to $z$, the function $p(z):=\frac{1}{z-z'}$, $z\in\bar\Di$, lies in $A(\Di)$ and satisfies
 \[
 |p(z')|>\max_{z\in K}|p(z)|.
 \]
 This implies that the polynomially convex hull $\hat K$ of $K$ does not contain points from $\mathbb T\setminus L$. Thus $\hat K$ is a compact subset of $L\cup\Di$ as well. Further, the function $h^{-1}$ is continuous on $L\cup\Di$ and holomorphic on $\Di$. Thus by the Mergelyan theorem \cite{M}, $h^{-1}|_{\hat K}$ can be uniformly approximated by holomorphic polynomials. Hence, $\tilde g:=h^{-1}\circ g'$ can be uniformly approximated on $X$ by holomorphic polynomials in $g'$, i.e., it lies in $A$, as claimed. 
 
 Now, $\tilde g(X)$ is a compact subset of $\bar M_u\subset\mathbb\Di\, (=:N_u)$. Hence, $\tilde g(\mathfrak M(A))$ is a compact subset of $\mathbb D$. In particular, the map $g:=r\circ\tilde g:\mathfrak M(A)\rightarrow N$ lies in $\mathcal O_\mathcal F(\mathfrak M(A), M)$, where $\mathcal F:=\{\tilde g\}$, and $g|_S=f$, $g(S^c)\subset M$.
 
 The proof of the theorem is complete.\hfill $\Box$

 \sect{Proofs of Propositions \ref{prop1.4} and \ref{prop1.5} }
 \subsect{Proof of Proposition \ref{prop1.4}}
 Parts (1) and (2) follow directly from the definition of a universal manifold.\smallskip
 
 (3) Since $M\subset\Co^n$ is universal and $F: \Co^n\rightarrow\Co^m$ is a holomorphic embedding, in order to prove that $F(M)$ is universal it suffices to check that if $g\in A(X, \bar M)$, then $F\circ g\in A(X,F(\bar M))$. In fact, let $\widehat{g(X)}\Subset\Co^n$ be the polynomially convex hull of the compact set $g(X)\subset \Co^n$. By the Runge approximation theorem, see, e.g., \cite[I.F]{GR},  coordinates of $F$ are uniformly approximated on a neighbourhood of $\widehat{g(X)}$ by holomorphic polynomials. Since $A$ is a uniform algebra, this implies the required statement.\smallskip
 
 (4) Let $\bar{\mathbb B}^n\subset\RR^n$ be the closed unit Euclidean ball and $\mathbb S^{n-1}\subset \bar{\mathbb B}^n$ be the unit sphere. Since $\mathbb S^{n-1}$ is a peak-interpolation set for the algebra $C(\bar{\mathbb B}^n)$, by the definition of a universal manifold every  $f\in C(\mathbb S^{n-1},\bar M)$ extends to a  $g\in C(\bar{\mathbb B}^n,\bar M)$. This shows that all homotopy groups of $\bar M$ are trivial. In turn, since $M$ is a paracompact Banach manifold, the latter implies that $M$ is contractible, see the corollary after \cite[Thm.\,15]{Pa}.\hfill $\Box$
 \subsect{Proof of Proposition \ref{prop1.5}}
 First, we prove that under condition \eqref{equ1.4} the Minkowski functional $p_M: B\rightarrow [0,\infty)$ is continuous, i.e., for every $v\in B$ and every $\{v_n\}_{n\in\N}\subset B$ converging to $v$
  \[
 \lim_{n\rightarrow\infty}p_M(v_n)=p_M(v).
 \]

 To this end, for $v\in B$ we set 
 \[
 \overrightarrow{0v}:=\{tv\in B\, :\, t\in\RR_+\}.
 \]
Due to \eqref{equ1.4} for $v\ne 0$ there is some $t(v)\in  (0,\infty]$ such that  
\[
[0, t(v) v)\subset M,\quad \overrightarrow{0v}\setminus  [0, t(v) v)\not\subset M\quad {\rm and}\quad t(v) v\in\partial M\quad {\rm if}\quad t(v)<\infty.
\] 
In particular, $p_M(v)=\frac{1}{t(v)}$.
 
 Let $\{v_n\}_{n\in\N}\subset B\setminus\{0\}$ be a sequence converging to $v$.   
  If $v=0$, then 
 since $M$ is open, $\lim_{n\rightarrow\infty}t(v_n)=\infty$. Thus, 
 \[
 \lim_{n\rightarrow\infty}p_M(v_n)=\lim_{n\rightarrow\infty}\frac{1}{t(v_n)}=0=p_M(0).
 \]
 
 If $v\ne 0$, then  $tv\in M$ for every $t\in [0,t(v))$. Since $\{tv_n\}_{n\in\N}$ converges to $tv$ and $M$ is open, there is some $n(t)\in\N$ such that $tv_n\in M$ for all $n\ge n(t)$. This implies that
 \begin{equation}\label{e2.3}
 \varlimsup_{n\rightarrow\infty}p_M(v_n)=\varlimsup_{n\rightarrow\infty}\frac{1}{t(v_n)}\le \inf_{t\in [0,t(v))}\frac{1}{t}=p_M(v).
 \end{equation}
If $p_M(v)=0$, the latter implies that $\lim_{n\rightarrow\infty}p_M(v_n)=p_M(v)$. For otherwise, $t(v)<\infty$. Hence,  $tv\in \bar M^c$ for every $t>t(v)$. Since $\{tv_n\}_{n\in\N}$ converges to $tv$ and $\bar M^c$ is open, there is some $n(t)\in\N$ such that $tv_n\in \bar M^c$ for all $n\ge n(t)$. This implies that
\begin{equation}\label{e2.4}
p_M(v)=\sup_{t> t(v)}\frac{1}{t}\le\varliminf_{n\rightarrow\infty}\frac{1}{t(v_n)}=
\varliminf_{n\rightarrow\infty}p_M(v_n).
\end{equation}
From \eqref{e2.3}, \eqref{e2.4} we obtain that $\lim_{n\rightarrow\infty}p_M(v_n)=p_M(v)$.
Thus, $p_M$ is a continuous function.\smallskip

Now, assume that $p: B\rightarrow [0,\infty)$ is a continuous homogeneous function. Let
$M:=\{v\in B\, :\, p(v)<1\}$. Then $M$ is an open set containing $0$, i.e., $M$ is absorbing. Further, if $v\in\bar M$, then continuity and homogeneity of $p$ imply that $p(v)\le 1$  and $p(tv)\in M$ for all $t\in [0,1)$. Hence, $M$ satisfies condition \eqref{equ1.4}.

Finally, for $v\ne 0$
\[
p_M(v):=\inf_{tv\in M,\, t>0}\,\frac{1}{t}=\inf_{p(tv)<1,\, t>0}\,\frac{1}{t}=\frac{1}{\frac{1}{p(v)}}=p(v),
\]
as required.\hfill $\Box$

 \sect{Proofs of Theorem \ref{te1.4} and Corollary \ref{cor1.7}}
\subsect{}
This part contains some results used in the proof of Theorem \ref{te1.4}. \smallskip

Let $\theta:[0,1]\rightarrow [0,\frac{\pi}{4}]$ be a continuous function positive on $(0,1)$ and equal to zero at $\{0,1\}$ and let
\begin{equation}\label{equ3.1}
\Omega:=\{z=re^{i\theta}\in\mathbb C\, :\, 0<\theta<\theta(r),\ r\in (0,1)\}.
\end{equation}

Let $A$ be a uniform algebra on $X$ and let $S\subset X$ be a peak-interpolation set for $A$.
\begin{Lm}\label{lem3.1}
Given $\varepsilon\in (0,1)$ and a compact set $E\subset S^c$ there is a function $h_\varepsilon\in A$ such that 
\[
h_\varepsilon(X)\subset\bar{\Omega},\quad h_\varepsilon|_S=1,\quad |h_\varepsilon(x)|<1\quad \forall\, x\in S^c\quad{\rm and}\quad |h_\varepsilon(x)|\le\varepsilon\quad \forall\,x\in E.
\]
\end{Lm}
\begin{proof}
Since $\Omega$ is a simply connected domain whose boundary is the Jordan curve
\[
\gamma(t):=\left\{ 
\begin{array}{ccc}
2t&0\le t\le \frac 12\medskip\\
(2-2t)e^{i\theta(2-2t)}&\frac 12\le t\le 1,
\end{array}
\right.
\]
by the Carath\'eodory theorem, see, e.g., \cite{P}, there is a conformal map $\mathbb D\rightarrow\Omega$ that extends to a homeomorphism $G:\bar{\mathbb D}\rightarrow\bar{\Omega}$. Let $z_0:=G^{-1}(0), z_1:=G^{-1}(1)\in\mathbb T$.
Let $\chi\in A$ be such that $\chi|_S=1$ and $|\chi(x)|<1$ for all $x\in S^c$ (existing by the definition of a peak-interpolation set.) Then there exists $r\in (0,1)$ such that $\chi(E)\subset\mathbb D_r:=\{z\in\Co\, :\, |z|<r\}$. Consider the set of M\"{o}bius transformations of $\bar{\mathbb D}$:
\[
g_a(z):=\frac{z-az_0}{1-az_0^{-1}z},\quad  z\in\bar{\mathbb D},\quad -1<a<1.
\]
Then $g_a(z_0)=z_0$ for all $a$ and  $\lim_{a\rightarrow-1} g_a(\mathbb D_r)=\{z_0\}$ (convergence in the Hausdorff metric). In particular, there exists $a_\varepsilon\in (-1,0)$ such that $G\circ g_{a_\varepsilon}$ maps $\mathbb D_r$ into $\Omega\cap \mathbb D_\varepsilon$. This map sends $g_{a_\varepsilon}^{-1}(z_1)$ to $1$. Consider the function $\chi_\varepsilon:=g_{a_\varepsilon}^{-1}(z_1)\chi\in A$. Since $|g_{a_\varepsilon}^{-1}(z_1)|=1$, we have  $\chi_\varepsilon|_S=g_{a_\varepsilon}^{-1}(z_1)$, $\chi_\varepsilon(S^c)\subset\mathbb D$ and $\chi_\varepsilon(E)\subset\mathbb D_r$. We set $h_\varepsilon:=G\circ g_{a_\varepsilon}\circ\chi_\varepsilon$. Since $G\circ g_{a_\varepsilon}\in A(\mathbb D)$, it is a uniform limit of a sequence of holomorphic polynomials. Hence, $h_\varepsilon\in A$ and has the required properties.
\end{proof}
\begin{Lm}\label{lem3.2}
Let $\Psi:[0,1]\rightarrow \mathbb R_+$ be a continuous strictly increasing function equal to $0$ at $0$. There exists a sequence of functions $\{\psi_i\}_{i\in\mathbb N}\subset C([0,1])$ positive on $(0,1)$ and equal to zero at $\{0,1\}$ such that
\[
\sum_{i=1}^k \psi_i(r_i)\le\Psi\left(\sum_{i=1}^k\frac{r_i}{2^i}\right),\quad r_i\in [0,1],\ 1\le i\le k,\ k\in\mathbb N;
\]
here the equality holds if and only if all $r_i=0$. 
\end{Lm}
\begin{proof}
We set $\psi_1(r_1):=(1-r_1)\Psi(\frac{r_1}{2})\, (\le \Psi(\frac{r_1}{2}))$, $r_1\in [0,1]$. Suppose the required $\psi_i$ are already defined for all $i\le k-1$. Then we define
\[
\psi_k(r_k):=(1-r_k)\cdot\min_{r_1,\dots, r_{k-1}\in [0,1]}\left\{\Psi\left( \sum_{i=1}^{k} \frac{r_i}{2^i}\right)-\sum_{i=1}^{k-1}\psi_i(r_i)\right\},\quad r_k\in [0,1].
\]
By the induction hypothesis, $\psi_k$ is continuous equal to $0$ at $\{0,1\}$ and (since $\Psi$ is strictly increasing) for $r_k\not\in \{0,1\}$
\[
\psi_k(r_k)>(1-r_k)\cdot\min_{r_1,\dots, r_{k-1}\in [0,1]}\left\{\Psi\left( \sum_{i=1}^{k-1} \frac{r_i}{2^i}\right)-\sum_{i=1}^{k-1}\psi_i(r_i)\right\}= 0.
\]
I.e., $\psi_k$ is positive on $(0,1)$. 

Next, if one of $r_i\ne 0$, $1\le i\le k$, then by the induction hypothesis
\[
\psi_k(r_k)\le (1-r_k)\left(\Psi\left( \sum_{i=1}^{k} \frac{r_i}{2^i}\right)-\sum_{i=1}^{k-1}\psi_i(r_i)\right)<
\Psi\left( \sum_{i=1}^{k} \frac{r_i}{2^i}\right)-\sum_{i=1}^{k-1}\psi_i(r_i).
\]
This proves the required statement.
\end{proof}
\subsect{Proof of Theorem \ref{te1.4}}
Let $f\in C(S,\bar{M})$, $f(S)\cap\partial M\ne\emptyset$. According to \cite{S} there is a map $g\in A(X,B)$ such that  $g|_S=f$. Consider the compact set $K:=g(X)\subset B$.
By Mazur's theorem, see, e.g., \cite[Ch.\,VI,\,4.8]{Co}, the closure of  the convex balanced hull of $K$,
\[
\hat K:={\rm cl}\left\{\sum_{i=1}^n c_i v_i\, ,\, v_i\in K,\, c_i\in\mathbb D,\,\sum_{i=1}^n |c_i|=1,\, n\in\mathbb N \right\}
\]
is compact. Let $V\subset B$ be the subspace generated by vectors in $\hat K$ equipped with norm $\|\cdot\|_B$ and $M|_V:=M\cap V$. Then $M|_V$ is an open subset of $V$ and its Minkowski functional (defined on $V$) coincides with $p_{M}|_V$, i.e., it is continuous. Consider the modulus of continuity of $p_{M}|_{\hat K}$,
\[
\omega_{ p_{M}|_{\hat K}}(t):=\sup\{|p_{M}(v)-p_{M}(v')|\, :\  v,v'\in \hat K,\ \|v-v'\|_B\le t\},\quad t\ge 0.
\]
Since $K$ is compact and convex, $\omega_{ p_{M}|_{\hat K}}:[0,{\rm diam}\,\hat K]\rightarrow [0,\infty)$ is a nondecreasing continuous subadditive function equal to zero at $0$. We set
\begin{equation}\label{eq3.2}
\omega(t):=t+\omega_{ p_{M}|_{\hat K}}(t),\quad t\in [0,{\rm diam}\, \hat K].
\end{equation}
Then $\omega$ is a strictly increasing continuous subadditive function. Moreover, since $p_{M}|_{\hat K}$ attains value $1$ on $K$ and equals $0$ at $0$, the range of $\omega$ contains interval $[0,1]$. Thus the inverse $\omega^{-1}: [0,T]\rightarrow [0,{\rm diam}\, \hat K]$, $T:=\omega({\rm diam}\, \hat K)\ge 1$,  of $\omega$ is an increasing continuous function equals $0$ at $0$.

We define the required extension $h\in A(X,\bar M)$ of $f$  by the formula
\[
h(x)=\sum_{k=1}^\infty \frac{1}{2^k}h_k(x)g(x),\quad x\in X,
\]
for some $h_k\in A$, $|h_k|\le 1$. Here $h_k$ maps $X$ in the closure of a domain
\[
\Omega_k:=\{z=re^{i\theta}\in\mathbb C\, :\, 0<\theta<\theta_k(r),\ r\in (0,1)\},
\]
where $\theta_k:[0,1]\rightarrow[0,\frac{\pi}{4}]$ is a continuous function positive on $(0,1)$ and equals zero at $\{0,1\}$. 

To this end, we set
\[
\mathfrak m:=\max_{x\in X}\|g(x)\|_B,\qquad \mathfrak m':=\max_{x\in X}p_M (g(x))
\]
and choose continuous functions $\tilde\theta_k: [0,1]\rightarrow\mathbb R_+$ positive on $(0,1)$ and equals  $0$ at $\{0,1\}$ such that for every $n\in\N$
\begin{equation}\label{eq3.3}
\sum_{k=1}^n \tilde\theta_k(r_k)\le \frac{1}{2}\omega^{-1}\left( \sum_{k=1}^n \frac{r_k}{2^k}\right),\quad  r_k\in [0, 1],\  k\in\mathbb N.
\end{equation}
This is possible due to Lemma \ref{lem3.2}.

Next, we define
\begin{equation}\label{eq3.4}
\theta_k(r_k):=\min\left\{\frac{2^k}{\mathfrak m}\,\tilde\theta_k(1-r_k),\frac{\pi}{4}\right\},\quad r_k\in [0,1],\ k\in\mathbb N.
\end{equation}
Also, we define a sequence $\{\varepsilon_n\}_{n\in\mathbb Z_+}$ of positive numbers converging to $0$ by the formulas
\begin{equation}\label{eq3.5}
\varepsilon_0:=1,\qquad \varepsilon_{n+1}:=\min\left\{\varepsilon_n, \frac{2^n}{\mathfrak m}\omega^{-1}\left(\frac{\varepsilon_n}{\max\{\frac{\mathfrak m'}{\mathfrak m},1\}}\right),\frac{1}{2^{n+2}}\right\},\qquad n\ge 0.\smallskip
\end{equation}

Fix a proper open neighbourhood $U\subset X$ of $S$.
Let 
\begin{equation}\label{un}
U_n:=\{x\in U\, :\, p_M(g(x))<1+\varepsilon_n\},\quad n\in\mathbb N.
\end{equation}
Then $U_n\subseteq U$ is an open neighbourhood of $S$. We choose $h_n\in A$ with image in $\bar{\Omega}_n$ such that $h_n(S)=1$, $|h_n|_{S^c}|<1$ and $|h_n(x)|\le\varepsilon_n$ for all $x\in U_n^c\, (\ne\emptyset)$ (see Lemma \ref{lem3.1}). Note that $U_1\supseteq U_2\supseteq\cdots\supseteq U_n\supseteq\cdots $ because the sequence $\{\varepsilon_n\}_{n\in\N}$ is nonincreasing.  We set for convenience  $U_0:=X$, $h_0:=0$ and $\theta_0:=0$.

Since maps $h$ and $h'$,
\[
h(x)=\sum_{k=1}^\infty \frac{1}{2^k}h_k(x)g(x)\quad {\rm and}\quad h'(x):=\sum_{k=1}^\infty \frac{1}{2^k}|h_k(x)|g(x),\quad x\in X,
\]
map $X$ into $\hat K$,  
\begin{equation}\label{eq3.6}
\begin{array}{l}
\displaystyle
\|h(x)-h'(x)\|_B=
\left\| \,\sum_{k=1}^\infty \frac{1}{2^k}|h_k(x)|(e^{i{\rm Arg}(h_k(x))}-1) g(x)\right\|_B\le\smallskip\\
\displaystyle
 \sum_{k=1}^\infty  \frac{\mathfrak m}{2^k}|h_k(x)|\,2 \sin\left(\frac{{\rm Arg}(h_k(x)}{2}\right)
 \displaystyle \le \sum_{k=1}^\infty \frac{\mathfrak m}{2^k}|h_k(x)|\theta_k(|h_k(x)|)\le\smallskip\\
 \displaystyle
 \sum_{k=1}^\infty \frac{\mathfrak m}{2^k}\theta_k(|h_k(x)|)\le {\rm diam}\, \hat K,\quad x\in X.
\end{array}
\end{equation}

Suppose $x\in U_n\setminus U_{n+1}$, $n\in\mathbb Z_+$. Then \eqref{eq3.6},
subadditivity of $\omega_{ p_{M}|_{\hat K}}$  and \eqref{eq3.5} imply
\begin{equation}\label{eq3.7}
\begin{array}{l}
\displaystyle
p_M(h(x))\le p_M(h'(x))+\omega_{ p_{M}|_{\hat K}}(\|h(x)-h'(x)\|_B)\le\medskip\\
\displaystyle \left(\,\sum_{k=0}^\infty \frac{1}{2^k}|h_k(x)|\right) p_M(g(x))+
\omega_{ p_{M}|_{\hat K}}\!\!\left(\sum_{k=0}^\infty \frac{\mathfrak m}{2^k}|h_k(x)|\theta_k(|h_k(x)|)  \right)\le\medskip\\
\displaystyle   \left[\left(\sum_{k=0}^n \frac{1}{2^k}|h_k(x)|\right)(1+\varepsilon_n)+\omega_{ p_{M}|_{\hat K}}\!\!\left(\sum_{k=0}^n \frac{\mathfrak m}{2^k}|h_k(x)|\theta_k(|h_k(x)|)  \right)\right]+\medskip\\
\displaystyle \left[\left(\sum_{k=n+1}^\infty \frac{1}{2^k}|h_k(x)|\right) p_M(g(x))+
\omega_{ p_{M}|_{\hat K}}\!\!\left(\sum_{k=n+1}^\infty \frac{\mathfrak m}{2^k}|h_k(x)|\theta_k(|h_k(x)|)  \right)\right]\le\medskip
\\
\displaystyle  \left[\left(\sum_{k=0}^n \frac{1}{2^k}|h_k(x)|\right)(1+\varepsilon_n)+\omega\!\left(\sum_{k=0}^n \frac{\mathfrak m}{2^k}\theta_k(|h_k(x)|)  \right)\right]+\medskip\\
\displaystyle
\left[\frac{\mathfrak m'\varepsilon_{n+1}}{2^n}
+ \omega_{ p_{M}|_{\hat K}}\!\!\left(\frac{\mathfrak m\,\varepsilon_{n+1}}{2^n}  \right)\right].
\end{array}
\end{equation}
Note that due to \eqref{eq3.5}
\begin{equation}\label{eq3.8}
\frac{\mathfrak m'\varepsilon_{n+1}}{2^n}
+ \omega_{ p_{M}|_{\hat K}}\!\!\left(\frac{\mathfrak m\,\varepsilon_{n+1}}{2^n}  \right)\le \max\left\{\frac{\mathfrak m'}{\mathfrak m},1\right\}\cdot\omega\left(\frac{\mathfrak m\,\varepsilon_{n+1}}{2^n}  \right)\le\varepsilon_n.
\end{equation}
Also, we have for $n\ge 1$ (since $\varepsilon_n\le \frac{1}{2^{n+1}}$ and $|h_k|\le 1$)
\[
(1-\varepsilon_n)-\left(\sum_{k=1}^n \frac{1}{2^k}|h_k(x)|\right) (1+\varepsilon_n)\ge\sum_{k=1}^n \frac{1-|h_k(x)|}{2^k}+\frac{1}{2^n}-\varepsilon_n\!\left(2-\frac{1}{2^n}\right)> \sum_{k=1}^n \frac{1-|h_k(x)|}{2^k}.
\]
This and \eqref{eq3.3}, \eqref{eq3.4} imply for $n\ge 1$,
\begin{equation}\label{eq3.9}
\begin{array}{l}
\displaystyle
\left(\sum_{k=1}^n \frac{1}{2^k}|h_k(x)|\right)(1+\varepsilon_n)+\omega\!\left(\sum_{k=1}^n \frac{\mathfrak m}{2^k}\theta_k(|h_k(x)|)  \right)<\medskip\\
\displaystyle 1-\varepsilon_n +\omega\!\left(\sum_{k=1}^n \frac{\mathfrak m}{2^k}\theta_k(|h_k(x)|)  \right)-\sum_{k=1}^n \frac{1-|h_k(x)|}{2^k}\le 1-\varepsilon_n.
\end{array}
\end{equation}

Thus, applying estimates \eqref{eq3.8} and \eqref{eq3.9} to \eqref{eq3.7} we obtain that for all $x\in U_n\setminus U_{n+1}$, $n\in\mathbb Z_+$,
$p_M(h(x))< 1$, i.e., $h(x)\in M$ in this case.

Further, if $x\in\cap_{n\in\mathbb Z_+}U_n$, then $p_{M}(g(x))\le 1$.
Hence, as in \eqref{eq3.7} using continuity of $\omega$ and $\omega^{-1}$ and \eqref{eq3.3}, \eqref{eq3.4} we obtain
\begin{equation}\label{eq4.10}
\begin{array}{l}
\displaystyle
p_{M}(h(x))\le \left(\sum_{k=1}^\infty \frac{1}{2^k}|h_k(x)|\right)+\omega\left(\sum_{k=1}^\infty \frac{\mathfrak m}{2^k}\theta_k(|h_k(x)|)  \right)=\medskip\\
\displaystyle -\left(\sum_{k=1}^\infty \frac{1-|h_k(x)|}{2^k}\right)+1+\omega\left(\sum_{k=1}^\infty \frac{\mathfrak m}{2^k}\theta_k(|h_k(x)|)  \right)=\medskip\\
\displaystyle 1+\lim_{n\rightarrow\infty}\left(\omega\left(\sum_{k=1}^n \frac{\mathfrak m}{2^k}\theta_k(|h_k(x)|)  \right)-\sum_{k=1}^n \frac{1-|h_k(x)|}{2^k}\right)\le\medskip\\
\displaystyle 1+\omega\left(\frac{1}{2}\omega^{-1}\left(\sum_{k=1}^\infty\frac{1-|h_k(x)|}{2^k}\right)\right)-\sum_{k=1}^\infty \frac{1-|h_k(x)|}{2^k}
\le 1.
\end{array}
\end{equation}
Since $\frac 12\omega^{-1}(t)<\omega^{-1}(t)$ for $t>0$, the equality in \eqref{eq4.10} holds if and only if
\[
\sum_{k=1}^\infty\frac{1-|h_k(x)|}{2^k}=0.
\]
This implies that $|h_k(x)|=1$ for all $k\in\N$. In turn, according to our construction (see Lemma \ref{lem3.1}), the latter implies that $x\in S$. Thus, $p_M(h(x))< 1$ for $x\in \bigl(\cap_{n\in\mathbb Z_+}U_n\bigr)\setminus S$,  i.e., $h(x)\in M$ for such $x$ as well.

Therefore $h$ maps $X$ in $\bar{M}$,  $h|_S=f$ and $h(S^c)\subset M$, as required.
\hfill $\Box$
\subsect{Proof of Corollary \ref{cor1.7}}
The first statement follows directly from Theorem \ref{te1.4} (see also Remark \ref{rem1.7}) as  $[{\rm co}(f(S))]_\varepsilon\subset B$ is a bounded open convex set containing $f(S)$. The second one is a consequence of Theorem \ref{te1.4} as well as $({\rm co}(f(S)))^\circ\subset\Co^n$ is a bounded open convex set and $f(S)\cap\partial({\rm co}(f(S)))\ne\emptyset$.\hfill $\Box$

\end{document}